\documentclass[10pt,reqno]{amsart}

\usepackage[utf8]{inputenc}
\usepackage[english]{babel}
\usepackage{amsmath,amsfonts,amssymb,amsthm,shuffle}
\usepackage[T1]{fontenc}
\usepackage{lmodern}
\usepackage{mathtools}
\usepackage[titletoc]{appendix}

\usepackage[top=3.5cm,bottom=3.5cm,left=3.6cm,right=3.6cm]{geometry}

\usepackage[dvipsnames]{xcolor}
\usepackage[hyperindex=true,frenchlinks=true,colorlinks=true,
citecolor=Mahogany,linkcolor=DarkOrchid,urlcolor=Tan,linktocpage,
pagebackref=true]{hyperref}

\usepackage{tikz}
\usetikzlibrary{shapes}
\usetikzlibrary{fit}
\usetikzlibrary{decorations.pathmorphing}

\usepackage{dsfont}
\usepackage{wasysym}
\usepackage{stmaryrd}
\usepackage{cite}
\usepackage{subfigure}
\usepackage{multirow}
\usepackage{enumitem}
\usepackage{multicol}
\usepackage{bm}

\title[Convergence rates for weighted 
sums of Bernoulli random fields]{Convergence rates in the central limit theorem for weighted 
sums of Bernoulli random fields}
\date{\today}
\author{Davide Giraudo}
\address{Ruhr-Universität Bochum
Fakultät für Mathematik
NA 3/32
Universitätsstraße 150
44780 Bochum‚ Germany.}
\email{Davide.Giraudo@ruhr-uni-bochum.de}

\numberwithin{equation}{subsection}
\renewcommand{\leq}{\leqslant}
\renewcommand{\geq}{\geqslant}

\newtheorem{Theorem}{Theorem}[section]
\newtheorem{Th\'eor\`eme}{Th\'eor\`eme}[section]

\newtheorem{Lemma}[Theorem]{Lemma}

\newtheorem{Definition}[Theorem]{Definition}
\newtheorem{D\'efinition}[Th\'eor\`eme]{D\'efinition}
\newtheorem{Corollary}[Theorem]{Corollary}

\theoremstyle{remark}
\newtheorem{Remark}[Theorem]{Remark}
\newtheorem{Example}[Theorem]{Example}

\tikzstyle{Vertex}=[circle,draw=LimeGreen!80,fill=LimeGreen!8,
inner sep=1pt,minimum size=2mm,line width=1pt,font=\scriptsize]
\tikzstyle{Node}=[Vertex,draw=RoyalBlue!80,fill=RoyalBlue!8,inner sep=1.5pt]
\tikzstyle{Leaf}=[rectangle,draw=Black!70,fill=Black!16,
inner sep=0pt,minimum size=1mm,line width=1.25pt]
\tikzstyle{Edge}=[Maroon!80,cap=round,line width=1pt]
\tikzstyle{Mark1}=[draw=BrickRed!80,fill=BrickRed!8]
\tikzstyle{Mark2}=[draw=BurntOrange!80,fill=BurntOrange!8]
\tikzstyle{EdgeRew}=[->,RedOrange!80,cap=round,thick]

\newcommand \ens[1]{\left\{ #1\right\}}
\newcommand \R{\mathbb R}

\newcommand \N{\mathbb N}

\newcommand{\E}[1]{\mathbb E\left[#1\right]}

\newcommand \Z{\mathbb Z}

\newcommand \abs[1]{\left|#1\right|}
\newcommand \eps{\varepsilon}

\newcommand{\f}{\mathcal F}

\newcommand{\pr}[1]{\left(#1\right)}
\newcommand{\norm}[1]{\left\lVert #1 \right\rVert}
\newcommand{\gr}[1]{\bm{#1}}
\newcommand{\imd}{\preccurlyeq}

\begin{document}

%%%%%%%%%%%%%%%%%%%%%%%%%%%%%%%%%%%%%%%%%%%%%%%%%%%%%%%%%%%%%%%%%%%%%%%%
%%%%%%%%%%%%%%%%%%%%%%%%%%%%%%%%%%%%%%%%%%%%%%%%%%%%%%%%%%%%%%%%%%%%%%%%
%%%%%%%%%%%%%%%%%%%%%%%%%%%%%%%%%%%%%%%%%%%%%%%%%%%%%%%%%%%%%%%%%%%%%%%%
\begin{abstract}
 We prove moment inequalities for a class of functionals 
 of i.i.d. random fields. We then derive rates in the central limit theorem 
 for weighted sums of such randoms fields via an approximation 
 by $m$-dependent random fields. 
\end{abstract}
\maketitle 

\section{Introduction and main results}
  
\subsection{Goal of the paper}  

In its simplest form, the central limit theorem states that if $\pr{X_i}_{i\geq 1}$ 
is an independent identically distributed (i.i.d.) sequence of centered random variables 
having variance one, then the sequence $\pr{n^{-1/2}\sum_{i=1}^nX_i}_{n\geq 1}$ 
converges in distribution to a standard normal random variable. If $X_1$ has 
a finite moment of order three, Berry \cite{MR0003498} and Esseen 
\cite{MR0011909} gave the following convergence rate:
\begin{equation}
\sup_{t\in\R}\abs{\mathbb P\ens{ n^{-1/2}\sum_{i=1}^nX_i\leq t}-
 \mathbb P\ens{N\leq t} }\leq C\E{\abs{X_1}^3}n^{-1/2},
\end{equation}
where $C$ is a numerical constant and $N$ has a standard normal 
distribution. The question of extending the 
previous result to a larger class of sequence have received a lot 
of attention. When $X_i$ can be represented as a function of 
an i.i.d. sequence, optimal convergence rates are given in 
\cite{MR3502600}.

In this paper, we will focus on random fields, that is, collection 
of random variables indexed by $\Z^d$ and more precisely in
 Bernoulli random fields, which are defined as follows. 
 
 \begin{Definition}
  Let $d\geq 1$ be an integer. The random field 
  $\pr{X_{\gr{n}}}_{\gr{n}\in \Z^d}$ is said to be Bernoulli if 
  there exist an i.i.d. random field $\pr{\eps_{\gr{i}}}_{\gr{i}\in \Z^d}$ 
  and a measurable function $f\colon \R^{\Z^d}\to \R$ such that 
  $X_{\gr{n}}=f\pr{\pr{\eps_{\gr{n}-\gr{i}}}_{\gr{i}\in \Z^d}}$ 
  for each $\gr{n}\in \Z^d$.
 \end{Definition}
  We are interested in the asymptotic behavior of the sequence 
$\pr{S_n}_{n\geq 1}$ defined by 
\begin{equation} 
 S_n:=\sum_{\gr{i}\in\Z^d}b_{n,\gr{i}}X_{\gr{i}},
\end{equation}
where 
$b_n:=\pr{b_{n,\gr{i}}}_{\gr{i}\in \Z^d}$ is an element of $\ell^2\pr{\Z^d}$. Under 
appropriated conditions on the dependence of the random field $\pr{X_{\gr{i}}}_{\gr{i}\in 
\Z^d}$ 
and the sequence of weights $\pr{b_n}_{n\geq 1}$ that will be 
specified later, the sequence $\pr{S_n/\norm{b_n}_2}_{n\geq 1}$ converges 
in law to a normal distribution \cite{MR3483738}. The goal of this paper is to 
provide bounds of the type Berry-Esseen in order to give convergence 
rates in the central limit theorem.    

This type of question has been addressed for the so-called $\operatorname{BL}
\pr{\theta}$-dependent random fields \cite{MR2325308}, martingale 
differences random fields \cite{MR2167826}, positively and negatively dependent 
random fields \cite{MR1661688,MR1355526} and mixing random fields 
\cite{MR2310650,MR1082637}.

In order to establish this kind of results, we need several ingredients. First, we need 
convergence rates for $m$-dependent random fields. Second, a Bernoulli 
random field can be decomposed as the sum of an $m$-dependent random field 
and a remainder. The control of the contribution of the remainder is done 
by a moment inequality in the spirit of Rosenthal's inequality \cite{MR0271721}. 
 One of the main applications of such an inequality is the estimate of the 
 convergence rates in the central limit theorem for random fields that can be expressed 
 as a functional of a random field consisting of i.i.d. random variables. The method 
 consists in approximating the considered random 
 field by an $m$-dependent one, and in controlling the approximation with the help of  
 the established moment inequality.  In the one dimensional case, 
probability and moment inequalities have been established in 
\cite{MR3114713} for 
maxima of partial sums of Bernoulli sequences. The techniques 
used therein permit to derive results for weighted sums of 
such sequences.

The paper is organized as follows. In Subsections~\ref{subsec:background}, we give the material 
which is necessary to understand the moment inequality stated in 
Theorem~\ref{thm:moment_inequality_dependence_coefficients}. We then give 
the results on convergence rates 
in Subsection~\ref{subsec:mian_results} (for weighted sums, sums on 
subsets of $\Z^d$ and in a regression model)  and compare  the obtained results 
in the case of linear random fields with some existing ones. Section~\ref{sec:proofs} 
is devoted to the proofs.

\subsection{Backgroud}\label{subsec:background}

   The following version of Rosenthal's inequality  is
   due to Johnson, Schechtman and Zinn \cite{MR770640}: 
   if 
   $\pr{Y_i}_{i=1}^n$ are independent centered random variables with 
   a finite moment of order $p\geq 2$, then 
   
   \begin{equation}\label{eq:Rosenthal}
    \norm{\sum_{i=1}^nY_i}_p\leqslant \frac{14.5p}{\log p}\left(\left(
    \sum_{i=1}^n\norm{Y_i}_2^2\right)^{1/2}+
    \left(
    \sum_{i=1}^n\norm{Y_i}_p^p\right)^{1/p}\right),
   \end{equation}
   where $\norm{Y}_q:=\pr{\E{\abs{Y}^q  }}^{1/q}$ for $q\geq 1$.

   It was first estalish without explicit constant in Theorem~3 of \cite{MR0271721}.

  Various extension of Rosenthal-type inequalities have been 
  obtained under mixing conditions \cite{MR1334179,MR2117923} or 
  projective conditions \cite{MR2255301,MR2472010,MR3077530}.  
 We are interested by extensions of \eqref{eq:Rosenthal} to the 
 setting of dependent random fields.  
 
Throughout the paper, we shall use the following notations.
 
 \begin{enumerate}[label=(N.\arabic*)]
  \item For a positive integer $d$, the set $\ens{1,\dots,d}$ is denoted by $[d]$.
  \item The coordinatewise order is denoted by $\imd$, that is, 
   for $\gr{i}=\pr{i_q}_{q=1}^d\in \Z^d$ and 
  $\gr{j}=\pr{j_q}_{q=1}^d\in \Z^d$,
  $\gr{i}
  \imd \gr{j}$ means that $i_k\leq j_k$ for any $k\in [d]$. 
  \item For $k\in [d]$, $\gr{e_k}$ denotes the element of $\Z^d$ 
  whose $q$th coordinate is $1$ and all the others are zero. Moreover, we write 
  $\gr{0}=(0,\dots,0)$ and $\gr{1}=(1,\dots,1)$.
  \item For $\gr{n}=\pr{ n_k}_{k=1}^d\in \N^d$, we write the product $\prod_{k=1}^dn_q$ as 
  $\abs{\gr{n}}$.
  \item The cardinality of a set $I$ is denoted by $\abs{I}$.
  \item For a real number $x$, we denote by $\left[x\right]$ the 
  unique integer such that $[x]\leq x<[x]+1$.
  \item\label{not:standard_normal} We write $\Phi$ for the cumulative distribution function 
  of a standard normal law.
  \item\label{not:ensemble_translate} If $\Lambda$ is a 
  subset of $\Z^d$ and $\gr{k}\in \Z^d$, then $\Lambda-\gr{k}$ is 
  defined as $\ens{\gr{l}-\gr{k},\gr{l}\in \Lambda}$.
  \item For $q\geq 1$, we denote by $\ell^{q}\pr{\Z^d}$ the space of 
  sequences $\gr{a}:=\pr{a_{\gr{i}}}_{\gr{i}\in \Z^d}$ such that 
  $\norm{\gr{a}}_{\ell^q}:=\pr{\sum_{ \gr{i}\in \Z^d} \abs{a_{\gr{i}}}^q   }^{1/q}
  <+\infty$.
  \item For $\gr{i}=\pr{i_q}_{q=1}^d$, the quantity 
  $\norm{\gr{i}}_\infty$ is defined as $\max_{1\leq q\leq d}
  \abs{i_q}$.
 \end{enumerate}
 Let $\pr{Y_{\gr{i}}}_{\gr{i}\in \Z^d}$ be a random field.
 The sum $\sum_{\gr{i}\in \Z^d}Y_{\gr{i}}$ is understood as the 
 $\mathbb L^1$-limit of the sequence $\pr{S_k}_{k\geq 1}$ where 
 $S_k=\sum_{\gr{i}\in \Z^d,\norm{\gr{i}}_{\infty} \leq k}Y_{\gr{i}}$.

 Following \cite{MR2172215} we define 
 the physical dependence measure.
\begin{Definition}
 Let $\pr{X_{\gr{i}}}_{\gr{i}\in\Z^d}:=
\pr{f\pr{\pr{\eps_{\gr{i}-\gr{j}}} }_{\gr{j}\in\Z^d}  }_{\gr{i}\in \Z^d}$ be a Bernoulli random field, 
$p\geq 1$ and 
$\pr{\eps'_{\gr{u}}}_{\gr{u}
\in \Z^d}$ be an i.i.d. random field which is independent 
of the i.i.d. random field $\pr{\eps_{\gr{u}}}_{\gr{u}\in \Z^d}$ and
has the same distribution 
as $\pr{\eps_{\gr{u}}}_{\gr{u}\in \Z^d}$. For $\gr{i}\in \Z^d$, 
we introduce the physical dependence measure
\begin{equation}
 \delta_{\gr{i},p} :=\norm{X_{\gr{i}}-X^*_{\gr{i}}}_p
\end{equation}
where $X_{\gr{i}}^*=f\pr{  \pr{\eps^*_{\gr{i}-\gr{j}}}_{\gr{j}\in\Z^d}  }$ and 
$\eps^*_{\gr{u}}=\eps_{\gr{u}}$ if $\gr{u}\neq \gr{0}$, 
$\eps^*_{\gr{0}}=\eps'_{\gr{0}}$.
\end{Definition}

In \cite{MR2988107,MR3256190}, various examples of Bernoulli 
random fields are given, for which the physical dependence measure 
is either computed or estimated. Proposition~1 of \cite{MR2988107}
also gives the following moment inequality: if $\Gamma$ is 
a finite subset of $\Z^d$, $\pr{a_{\gr{i}}}_{\gr{i}\in\Gamma}$ 
is a family of real numbers and $p\geq 2$, then for any Bernoulli 
random field $\pr{X_{\gr{n}}}_{\gr{n}\in\Z^d}$, 
\begin{equation}\label{eq:inegalite_EM_V_W}
 \norm{\sum_{\gr{i}\in \Gamma}a_{\gr{i}}X_{\gr{i}}}_p
 \leq \pr{2p\sum_{\gr{i}\in\Gamma}a_{\gr{i}}^2}^{1/2}
 \cdot \sum_{\gr{j}\in \Z^d}\delta_{\gr{j},p}.
\end{equation}
This was used in \cite{MR2988107,MR3256190} in order to establish 
functional central limit theorems. Truquet \cite{MR2684016} 
also obtained an inequality in this spirit.
 If $\pr{X_{\gr{i}}}_{\gr{i}\in \Z^d}$ 
is i.i.d. and centered, \eqref{eq:Rosenthal} would give 
\begin{equation}
 \norm{\sum_{\gr{i}\in \Gamma}a_{\gr{i}}X_{\gr{i}}}_p
 \leq C\pr{ \sum_{\gr{i}\in\Gamma}a_{\gr{i}}^2}^{1/2}
 \norm{X_{\gr{1}}}_p,
\end{equation}
while Rosenthal's inequality \eqref{eq:Rosenthal} 
would give
\begin{equation}
 \norm{\sum_{\gr{i}\in \Gamma}a_{\gr{i}}X_{\gr{i}}}_p
 \leq C\pr{ \sum_{\gr{i}\in\Gamma}a_{\gr{i}}^2}^{1/2}
 \norm{X_{\gr{1}}}_2+
 C\pr{ \sum_{\gr{i}\in\Gamma}\abs{a_{\gr{i}}}^p}^{1/p}
 \norm{X_{\gr{1}}}_p,
\end{equation}
 a better result in this context.

In the case of linear processes, equality $\delta_{\gr{j},p}\leq 
K\delta_{\gr{j},2}$ holds for a constant $K$ which does not 
depend on $\gr{j}$. However, there are processes for which such an 
inequality does not hold. 

\begin{Example}
We give an example of a random field such that there is no constant $K$ such that 
$\delta_{\gr{j},p}\leq 
K\delta_{\gr{j},2}$ holds for all $\gr{j}\in\Z^d$.
 Let $p\geq 2$ and let $\pr{\eps_{\gr{i}}}_{\gr{i}\in \Z^d}$ be an i.i.d. random field 
 and for each $\gr{k}\in \Z^d$, let $f_{\gr{k}}\colon \R\to \R$ be 
 a function such that the random variable $Z_{\gr{k}}:=f_{\gr{k}}\pr{\eps_{\gr{0}}}$ 
 is centered and has a finite moment of order $p$, and 
 $\sum_{\gr{k}\in \Z^d}\norm{Z_{\gr{k}}}_2^2<+\infty$.  
 Define $X_{\gr{n}}:=\lim_{N\to+\infty}
 \sum_{-N\gr{1}\imd \gr{j}\imd N\gr{1}}f_{\gr{k}}\pr{\eps_{\gr{n}-\gr{k}}}
 $, where the limit is taken in $\mathbb L^2$. Then 
 $X_{\gr{i}}-X^*_{\gr{i}}=f_{\gr{i}}\pr{\eps_{\gr{0}}}-
 f_{\gr{i}}\pr{\eps'_{\gr{0}}}$ hence $\delta_{\gr{i},2}$ is of order 
 $\norm{Z_{\gr{i}}}_2$ while $\delta_{\gr{i},p}$ is of order
 $\norm{Z_{\gr{i}}}_p$.
\end{Example}
Consequently, having the $\ell^p$-norm instead of the $\ell^2$-norm 
of the $\pr{a_{\gr{i}}}_{\gr{i}\in \Gamma}$ is more suitable. 

  \subsection{Mains results}\label{subsec:mian_results}
  
We now give a Rosenthal-like inequality for weighted sums of 
Bernoulli random fields in terms 
of the physical dependence measure.

\begin{Theorem}\label{thm:moment_inequality_dependence_coefficients}
  Let $\ens{\eps_{\gr{i}},\gr{i}\in\Z^d}$ be an i.i.d. set of 
  random variables. Then for any measurable function $f\colon \R^{\Z^d}\to \R$
  such that $X_{\gr{j}}:=f\pr{\pr{X_{\gr{j}-\gr{i}}}_{\gr{i}\in \Z^d}}$ has a finite 
  moment of order $p\geq 2$ and is centered, and any $\pr{a_{\gr{i}}}_{\gr{i}\in 
  \Z^d}\in \ell^2\pr{\Z^d}$, 
  \begin{multline}\label{eq:main_result_moment_inequality_PDM}
   \norm{\sum_{\gr{i}\in \Z^d} a_{\gr{i}}X_{\gr{i}}}_p
   \leq 
  \frac{14.5p}{\log p}\pr{\sum_{\gr{i}\in \Z^d}a_{\gr{i}}^2}^{1/2}
   \sum_{j=0}^{+\infty}\pr{4j+4 }^{d/2}\norm{X_{\gr{0},j}}_{2}\\
   +  \frac{14.5p}{\log p} 
   \pr{\sum_{\gr{i}\in \Z^d}\abs{a_{\gr{i}}}^p}^{1/p}
   \sum_{j=0}^{+\infty}\pr{4j+4 }^{d\pr{1-1/p}}\norm{X_{\gr{0},j}}_{p},
  \end{multline}
 where for $j\geq 1$,
 \begin{equation}
 X_{\gr{0},j}=\E{X_{\gr{0}}\mid \sigma\ens{\eps_{\gr{u}},\norm{\gr{u}}_\infty\leq 
 j   }}-\E{X_{\gr{0}}\mid \sigma\ens{\eps_{\gr{u}},\norm{\gr{u}}_\infty\leq 
 j -1  }} 
 \end{equation}
 and $ X_{\gr{0},0}=\E{X_{\gr{0}}\mid \sigma\ens{\eps_{\gr{0}}}}$.
\end{Theorem}

We can formulate a version of inequality 
\eqref{eq:main_result_moment_inequality_PDM} where the right hand side is expressed in 
terms of the coefficients of physical dependence measure. The obtained result is not directly 
comparable to \eqref{eq:inegalite_EM_V_W} because of the presence of the $\ell^p$-norm 
of the coefficients.

\begin{Corollary}\label{cor:moment_inequality_Wu}
Let $\ens{\eps_{\gr{i}},\gr{i}\in\Z^d}$ be an i.i.d. set of 
  random variables. Then for any measurable function $f\colon \R^{\Z^d}\to \R$
  such that $X_{\gr{j}}:=f\pr{\pr{X_{\gr{j}-\gr{i}}}_{\gr{i}\in \Z^d}}$ has a finite 
  moment of order $p\geq 2$ and is centered, and any $\pr{a_{\gr{i}}}_{\gr{i}\in 
  \Z^d}\in \ell^2\pr{\Z^d}$, 
  \begin{multline}\label{eq:main_result_moment_inequality_Wu_coeff}
   \norm{\sum_{\gr{i}\in \Z^d} a_{\gr{i}}X_{\gr{i}}}_p
   \leq \sqrt 2
  \frac{14.5p}{\log p}\pr{\sum_{\gr{i}\in \Z^d}a_{\gr{i}}^2}^{1/2}
   \sum_{j=0}^{+\infty}\pr{4j+4 }^{d/2}\pr{\sum_{\norm{\gr{i}}_\infty=j}
   \delta_{\gr{i},2}^2}^{1/2}\\
   +  \sqrt 2\frac{14.5p}{\log p} \sqrt{p-1}
   \pr{\sum_{\gr{i}\in \Z^d}\abs{a_{\gr{i}}}^p}^{1/p}
   \sum_{j=0}^{+\infty}\pr{4j+4 }^{d\pr{1-1/p}}\pr{\sum_{\norm{\gr{i}}_\infty=j}
   \delta_{\gr{i},p}^2}^{1/2}.
  \end{multline}
\end{Corollary}

Let $\pr{X_{\gr{j}}}_{\gr{j}\in\Z^d}=
f\pr{\pr{\eps_{\gr{j}-\gr{i}}}_{\gr{i}\in\Z^d}}$ be a centered square 
integrable Bernoulli 
random field and for any positive integer $n$, let $
b_n:=\pr{b_{n,\gr{i}}}_{\gr{i}\in \Z^d}$ be an element of $\ell^2\pr{\Z^d}$.
We are interested in the asymptotic behavior of the sequence 
$\pr{S_n}_{n\geq 1}$ defined by 
\begin{equation}\label{eq:definition_de_Sn}
 S_n:=\sum_{\gr{i}\in\Z^d}b_{n,\gr{i}}X_{\gr{i}}.
\end{equation}
Let us denote for $\gr{k}\in \Z^d$ 
the map $\tau_{\gr{k}}\colon\ell^2\pr{\Z^d}\to \ell^2\pr{\Z^d}$ 
defined by $\tau_{\gr{k}}\pr{\pr{x_{\gr{i}}}_{\gr{i}\in \Z^d}}
:=\pr{x_{\gr{i}+\gr{k}}}_{\gr{i}\in \Z^d}$.

In \cite{MR3483738}, Corollary~2.6 gives the following result: under a
Hannan type condition on the random field $\pr{X_{\gr{i}}}_{\gr{i}\in \Z^d}$ and under the 
following condition on the weights:
for any $q\in [d]$, 
\begin{equation}\label{eq:condition_sur_b_n}
 \frac 1{\norm{b_n}_{\ell^2}}\norm{\tau_{\gr{e_q}}\pr{b_n}-b_n   }_{\ell^2}=0,
\end{equation}
the series $\sum_{\gr{i}\in \Z^d}\abs{\operatorname{Cov}
\pr{X_{\gr{0}},X_{\gr{i}}}}$ converges and with
\begin{equation}\label{eq:definition_de_sigma}
 \sigma:= \pr{
\sum_{\gr{i}\in \Z^d}\operatorname{Cov}
\pr{X_{\gr{0}},X_{\gr{i}}} }^{1/2}, 
\end{equation}
the sequence $\pr{S_n/\norm{b_n}_{\ell^2}}_{n\geq 1}$ converges 
in distribution to a centered normal distribution with variance 
$\sigma^2$. The argument relies on an approximation by an 
$m$-dependent random field. 

The purpose of the next theorem is to give a general speed of convergence. In order 
to measure it, we define 
\begin{equation}
 \Delta_n:=
 \sup_{t\in \R}\abs{\mathbb P\ens{\frac{S_n}{\norm{b_n}_{\ell^2}    }\leq t}
 -\Phi\pr{t/\sigma}  }.
\end{equation}
The following quantity will also play an important role in the 
estimation of convergence rates.
\begin{equation}\label{eq:definition_de_eps_n}
 \eps_n:=\sum_{\gr{j}\in \Z^d}\abs{\E{X_{\gr{0}}X_{\gr{j}}}}
 \sum_{\gr{i}\in\Z^d}\abs{\frac{b_{n,\gr{i}}b_{n,\gr{i}+\gr{j}}}{\norm{b_n}_{\ell^2}
 }-1}.
\end{equation}
\begin{Theorem}\label{thm:vitesse_de_convergence_sommes_ponderees}
 Let $p>2$, $p':=\min\ens{p,3}$ and let $\pr{X_{\gr{j}}}_{\gr{j}\in\Z^d}=
\pr{f\pr{\pr{\eps_{\gr{j}-\gr{i}}}_{\gr{i}\in\Z^d}}}_{\gr{j}\in\Z^d}$ be a centered  Bernoulli 
random field with a finite moment of order $p$ and for any positive integer $n$, let $
b_n:=\pr{b_{n,\gr{i}}}_{\gr{i}\in \Z^d}$ be an element of $\ell^2\pr{\Z^d}$ 
such that for any $n\geq 1$, the set $\ens{\gr{k}\in \Z^d, 
b_{n,\gr{k}}\neq 0}$ is finite and nonempty, $\lim_{n\to +\infty}
\norm{b_n}_{\ell^2}=+\infty$
 and \eqref{eq:condition_sur_b_n} holds 
for any $q\in [d]$. 
Assume that for some positive $\alpha$ and $\beta$, the following 
series are convergent: 
\begin{equation}\label{eq:definition_de_C2_et_Cp}
C_2\pr{\alpha}:=\sum_{i=0}^{+\infty}\pr{i+1}^{d/2+\alpha}
 \norm{X_{\gr{0},i}}_2\mbox{ and }C_p\pr{\beta}:=\sum_{i=0}^{+\infty}
 \pr{i+1}^{d\pr{1-1/p}+\beta}
 \norm{X_{\gr{0},i}}_p.
\end{equation}

Let $S_n$ be defined by \eqref{eq:definition_de_Sn},

Assume that $\sum_{\gr{i}\in \Z^d}\abs{\operatorname{Cov}
\pr{X_{\gr{0}},X_{\gr{i}}}}$ is finite and that  $\sigma$ be given by 
\eqref{eq:definition_de_sigma} is positive.
Let $\gamma>0$ and let 
\begin{equation}\label{eq:definition_de_n0}
 n_0:=\inf\ens{N\geq 1\mid \forall n\geq N,
 \sqrt{\sigma^2+\eps_n}-29\pr{\log 2}^{-1}
  C_2\pr{\alpha} \pr{\left[\norm{b_n}_{\ell^2}\right]^{\gamma}}^{-\alpha} 
  \geq \sigma/2   }.
\end{equation}
Then for each $n\geq n_0$, 
\begin{multline}\label{eq:Berry_Esseen_weighted_sums}
 \widetilde{\Delta_n}\leq 150\pr{29\pr{\left[\norm{b_n}_{\ell^2}\right]+21}^\gamma+21}^{\pr{p'-1}d}
 \norm{X_{\gr{0}}}_{p'}^{p'}
 \pr{\frac{\norm{b_n}_{\ell^{p'}} }{\norm{b_n}_{\ell^{2}} } }^{p'}
 \pr{\sigma/2  }^{-p'}\\+
 \pr{2\frac{\abs{\eps_n}}{\sigma^2}+80\pr{\log 2}^{-1}
 \frac{ \norm{b_n}_{\ell^2} ^{-\gamma\alpha}}{\sigma^2}C_2\pr{\alpha}^2}  
 \pr{2\pi e}^{-1/2}\\
 +\pr{ \frac{14.5p}{\sigma\log p} 4^{d/2}\norm{b_n}_{\ell^2}^{-\gamma\alpha}
 C_2\pr{\alpha}   }^{\frac{p}{p+1}}+
   \pr{\frac{\norm{b_n}_{\ell^p}}{\sigma \norm{b_n}_{\ell^2}   }\frac{14.5p}{\log p} 
   4^{d\pr{1-1/p}} \norm{b_n}_{\ell^2}^{-\gamma\beta}C_p\pr{\beta}
   }^{\frac{p}{p+1}}.
\end{multline}
In particular, there exists a constant $\kappa$ such that for all $n\geq n_0$, 
\begin{equation}\label{eq:Berry_Esseen_weighted_sums_sans_cst}
 \Delta_n\leq \kappa \pr{\norm{b_n}_{\ell^2}^{\gamma\pr{p'-1}d-p'}
  \norm{b_n}_{\ell^{p'}}^{p'} +
 \abs{\eps_n} 
 + \norm{b_n}_{\ell^2}^{-\gamma\alpha\frac{p}{p+1}}
 +\norm{b_n}_{\ell^p}^{\frac{p}{p+1}}\norm{b_n}_{\ell^2}^{-\frac{p}{p+1}
 \pr{\gamma\beta+1}}}.
\end{equation}
\end{Theorem}
 \begin{Remark}
  If  \eqref{eq:condition_sur_b_n},  
 $\lim_{n\to +\infty}
\norm{b_n}_{\ell^2}=+\infty$
 and the family $\pr{\delta_{\gr{i},2}}_{\gr{i}\in \Z^d}$ is summable,
 then the sequence $\pr{\eps_n}_{n\geq 1}$ converges to $0$ 
 hence $n_0$ is well-defined. However, it is not clear to us whether the finiteness 
 of $C_2\pr{\alpha}$ combined with \eqref{eq:condition_sur_b_n} and 
 $\lim_{n\to +\infty}\norm{b_n}_{\ell^2}=+\infty$ imply that 
 $\sum_{\gr{j}\in \Z^d}\abs{\E{X_{\gr{0}}X_{\gr{j}}}}$ is finite. Nevertheless, we can
 show an analogous result in terms of $\delta_{\gr{i},p}$ coefficients by changing the 
 following in the statement of Theorem~\ref{thm:vitesse_de_convergence_sommes_ponderees}:
 \begin{enumerate}
 \item the definition of $C_2\pr{\alpha}$ should be replaced by 
 \begin{equation}
 C_2\pr{\alpha}:=\sqrt 2\sum_{j=0}^{+\infty}\pr{j+1}^{d/2+\alpha}
 \pr{\sum_{\norm{\gr{i}}_\infty=j}\delta_{\gr{i},2}^2}^{1/2};
 \end{equation}
 \item the definition of $C_p\pr{\beta}$ should be replaced by 
 \begin{equation}
 C_p\pr{\beta}:=\sqrt{2\pr{p-1}}\sum_{j=0}^{+\infty}\pr{j+1}^{d\pr{1-1/p}+\beta}
 \pr{\sum_{\norm{\gr{i}}_\infty=j}\delta_{\gr{i},2}^2}^{1/2}.
 \end{equation}
 \end{enumerate}
 In this case, the convergence of  $\sum_{\gr{i}\in \Z^d}\abs{\operatorname{Cov}
\pr{X_{\gr{0}},X_{\gr{i}}}}$ holds (cf. Proposition~2 in \cite{MR2988107}).
 \end{Remark}

 Recall notation~\ref{not:ensemble_translate}. Let
 $\pr{\Lambda_n}_{n\geq 1}$ be a sequence of subsets of $\Z^d$. 
 The choice $b_{n,\gr{j}}=1$ if $\gr{j}\in \Lambda_n$ and 
 $0$ otherwise yields the following corollary for set-indexed 
 partial sums.

  \begin{Corollary}\label{thm:vitesse_convergence_TLC}
  Let $\pr{X_{\gr{i}}}_{\gr{i}\in\Z^d}$ be a centered Bernoulli random field 
  with a finite moment of order $p\geq 2$, $p':=\min\ens{p,3}$
  and let $\pr{\Lambda_n}_{n\geq 1}$ be a sequence of subset of 
  $\Z^d$ such that $\abs{\Lambda_n}\to +\infty$ and for any $\gr{k}\in \Z^d$, 
  $\lim_{n\to +\infty}\abs{\Lambda_n\cap \pr{\Lambda_n-\gr{k}}}/
  \abs{\Lambda_n}=1$.
  Assume that the series defined in \eqref{eq:definition_de_C2_et_Cp} are 
  convergent for some positive $\alpha$ and $\beta$, that $\sum_{\gr{i}\in \Z^d}
  \abs{\operatorname{Cov}
\pr{X_{\gr{0}},X_{\gr{i}}}}$ is finite and that $\sigma$ 
  defined by \eqref{eq:definition_de_sigma} is positive.
Let $\gamma>0$ and $n_0$ be defined by \eqref{eq:definition_de_n0}.
There exists a constant $\kappa$ such that for any $n\geq n_0$, 
  
  \begin{multline}\label{eq:vitesse_convergence_TLC}
   \sup_{t\in \R}\abs{
   \mathbb P\ens{\frac{\sum_{\gr{i}\in \Lambda_n}X_{\gr{i}  }}{ 
   \abs{\Lambda_n}^{1/2}
   } \leq t
   }
   -\Phi\pr{t/\sigma}
   } \\
   \leq \kappa \pr{ \abs{\Lambda_n}^{q }
   +\sum_{\gr{j}\in \Z^d}\abs{ \E{X_{\gr{0}}X_{\gr{j}}}} \abs{\frac{
   \abs{\Lambda_n\cap \pr{\Lambda_n-\gr{j}} }    }
 { \abs{ \Lambda_n }
 }-1 }
 },
  \end{multline}
 where 
 \begin{equation}
  q:=\max\ens{\frac{\gamma\pr{p'-1}d-p'}2+1;-\gamma\alpha\frac{p}{2\pr{p+1}} 
  ;\frac{2-p-p\gamma\beta}{2\pr{p+1}}}.
 \end{equation}
  \end{Corollary}

 We consider now the following regression model:
 \begin{equation}
  Y_{\gr{i}}=g\pr{\frac{\gr{i}}n}+X_{\gr{i}}, \quad 
  \gr{i}\in\Lambda_n:=\ens{1,\dots,n}^d,
 \end{equation}
 where $g\colon [0,1]^d\to \R$ is an unknown smooth function 
 and $\pr{X_{\gr{i}}}_{\gr{i}\in \Z^d}$ is a zero mean 
 stationary Bernoulli random field. Let $K$ be a probability 
 kernel defined on $\R^d$ and let $\pr{h_n}_{n\geq 1}$ be a 
 sequence of positive numbers which converges to zero and which 
 satisfies 
 \begin{equation}\label{eq:assumption_nhn}
  \lim_{n\to+\infty}nh_n=+\infty\mbox{ and }\lim_{n\to+\infty}nh_n^{d+1}=0.
 \end{equation}

 We estimate the function 
 $g$ by the kernel estimator $g_n$ defined by 
 \begin{equation}\label{eq:definition_de_gnx}
 g_n\pr{\gr{x}}=\frac{\sum_{\gr{i}\in\Lambda_n}Y_{\gr{i}} 
 K\pr{\frac{\gr{x}-\gr{i}/n}{h_n}}}
 {\sum_{\gr{i}\in\Lambda_n} K\pr{\frac{\gr{x}-\gr{i}/n}{h_n}}}
 ,\quad x\in [0,1]^d.
 \end{equation}
 We make the following assumptions on the regression function $g$ 
 and the probability kernel $K$:
 
 \begin{enumerate}[label=(A)]
  \item\label{itm:assumption1} The probability kernel $K$ 
  fulfills $ \int_{\R^d}K\pr{\gr{u}}\mathrm d\gr{u}=1$, 
  is symmetric, non-negative, supported by $[-1,1]^d$. Moreover, there 
  exist positive constants $r$, $c$ and $C$ such that for any 
  $\gr{x},\gr{y}\in [-1,1]^d$, 
  $\abs{K\pr{\gr{x}}-K\pr{\gr{y}}} \leq r\norm{\gr{x}-\gr{y}}_\infty$ and 
  $c\leq K\pr{\gr{x}}\leq C$.
 
 \end{enumerate}
 
We measure the speed of convergence of $\pr{\pr{nh_n}^{d/2}\pr{g_n(\gr{x})-\E{g_n(\gr{x})}}}_{n\geq 1}$ to a normal distribution by the use of the quantity

 \begin{equation}
 \widetilde{\Delta_n}:=
 \sup_{t\in \R}\abs{\mathbb P\ens{\pr{nh_n}^{d/2}\pr{g_n(\gr{x})-\E{g_n(\gr{x})}} \leq t}
 -\Phi\pr{\frac{t}{\sigma\norm{K}_2   }}  }.
\end{equation}
Two other quantities will be involved, namely, 
\begin{equation}\label{eq:definition_de_An}
 A_n:=\pr{nh_n}^{d/2}
 \pr{\sum_{\gr{i}\in\Lambda_n} K^2\pr{\frac{\gr{x}-\gr{i}/n}{h_n}}    }^{1/2}
 \norm{K}_{\mathbb L^2\pr{\R^d}}^{-1}
 \pr{\sum_{\gr{i}\in\Lambda_n} K\pr{\frac{\gr{x}-\gr{i}/n}{h_n}}    }^{-1/2}
 \mbox{ and }
\end{equation}
\begin{equation}\label{eq:definition_de_eps_n_regression}
 \eps_n:=\sum_{\gr{j}\in \Z^d}\abs{\E{X_{\gr{0}}X_{\gr{j}}}}
 \pr{\sum_{\gr{i}\in\Lambda_n\cap  \pr{\Lambda_n-\gr{j}}}
 \frac{K\pr{\frac{\gr{x}-\gr{i}/n}{h_n}}
 K\pr{\frac{\gr{x}-\pr{\gr{i}-\gr{j}}/n}{h_n}}}{   \sum_{\gr{k}\in \Lambda_n}
 K^2\pr{\frac{\gr{x}-\gr{k}/n}{h_n}}
 }-1}.
\end{equation}

 \begin{Theorem}\label{thm:regression}
 Let $p>2$, $p':=\min\ens{p,3}$ and let $\pr{X_{\gr{j}}}_{\gr{j}\in\Z^d}=
\pr{f\pr{\pr{\eps_{\gr{j}-\gr{i}}}_{\gr{i}\in\Z^d}}}_{\gr{j}\in\Z^d}$ be a 
centered  Bernoulli 
random field with a finite moment of order $p$.
Assume that for some positive $\alpha$ and $\beta$, the following 
series are convergent: 
\begin{equation}\label{eq:definition_de_C2_et_Cp_bis}
 C_2\pr{\alpha}:=\sum_{i=0}^{+\infty}\pr{i+1}^{d/2+\alpha}
 \norm{X_{\gr{0},i}}_2\mbox{ and }C_p\pr{\beta}:=\sum_{i=0}^{+\infty}
 \pr{i+1}^{d\pr{1-1/p}+\beta}
 \norm{X_{\gr{0},i}}_p.
\end{equation}

Let $g_n\pr{\gr{x}}$ be defined by \eqref{eq:definition_de_gnx}, $\pr{h_n}_{n\geq 1}$ 
be a sequence which converges to $0$ and satisfies \eqref{eq:assumption_nhn},

Assume that $\sum_{\gr{i}\in \Z^d}\abs{\operatorname{Cov}
\pr{X_{\gr{0}},X_{\gr{i}}}}$ is finite and that $\sigma:=\sum_{\gr{j}\in \Z^d}\operatorname{Cov}
\pr{X_{\gr{0}},X_{\gr{j}}}>0$. 
Let $n_1\in \N$ be such that for each $n\geq n_1$, 
\begin{equation}\label{eq:}
\frac 12 \leq \pr{nh_n}^{-d} K\pr{\frac{\gr{x}-\gr{i}/n}{h_n}} \leq \frac 32\mbox{ and }
\end{equation}
\begin{equation}
 \frac 12\norm{K}_{\mathbb L^2\pr{\R^d}}
 \leq \pr{nh_n}^{-d} K^2\pr{\frac{\gr{x}-\gr{i}/n}{h_n}} \leq \frac 32
 \norm{K}_{\mathbb L^2\pr{\R^d}}.
\end{equation}
Let $n_0$ be the smallest integer for which for all $n\geq n_0$, 
 \begin{equation}
 \sqrt{\sigma^2+\eps_n}-29\pr{\log 2}^{-1}
  C_2\pr{\alpha} \pr{\left[\pr{\sum_{\gr{i}\in\Lambda_n}K\pr{\frac{1}{h_n} 
  \pr{\gr{x}-\frac{\gr{i}}n}}^2}^{1/2}     \right]^{\gamma}}^{-\alpha} 
  \geq \sigma/2 .
\end{equation}

Then there exists a constant $\kappa$ 
such that for each $n\geq \max\ens{n_0,n_1}$,
  \begin{multline}\label{eq:convergence_rates_regression}
   \Delta_n\leq \kappa \abs{A_n-1}^{\frac{p}{p+1}}+\abs{\eps_n}
   +\kappa\pr{nh_n}^{\frac{d}{2}\pr{\gamma\pr{p'-1}d-p'+2  }}\\
   +\pr{nh_n}^{-\frac{d}{2}\gamma\alpha\frac{p}{p+1}}
   +\pr{nh_n}^{\frac{2d-p\pr{\gamma\beta+1}   }{2\pr{p+1}} }.
  \end{multline} 
 
 \end{Theorem}

 Lemma~1 in \cite{MR2738877} shows that under 
 \eqref{eq:assumption_nhn},
 the sequence $\pr{A_n}_{n\geq 1}$ goes 
 to $1$ as $n$ goes to infinity and that the integer $n_1$ is 
 well-defined.
 
 We now consider the case of linear random fields in dimension 
 $2$, that is, 
 \begin{equation}
 X_{j_1,j_2}=\sum_{i_1,i_2\in\Z}a_{i_1,i_2} \eps_{j_1-i_1,j_2-i_2 },
 \end{equation}
  where $\pr{  a_{i_1,i_2} }_{i_1,i_2\Z}\in\ell^1\pr{\Z^2}$ and 
  $\pr{\eps_{u_1,u_2}}_{u_1,u_2\in\Z^2}$ is i.i.d., centered and $\eps_{0,0}$ 
  has a finite variance. We will focus on the case where the weights are of the form 
  $b_{n,i_1,i_2}=1$ if $1\leq i_1,i_2\leq n$ and $b_{n,i_1,i_2}=0$ otherwise.
 
Mielkaitis and Paulauskas \cite{MR2805741} established the following convergence rate. 
Denoting 
\begin{equation}
\Delta'_n:=
\sup_{r\geq 0}\abs{  \mathbb P
\ens{ \abs{\frac 1n\sum_{i_1,i_2=1}^nX_{i_1,i_2}}\leq r    }
-  \mathbb P
\ens{ \abs{N}\leq r    }
}
\end{equation}
and assuming that $\E{\abs{\eps_{0,0}}^{2+\delta} }$ is finite and 
\begin{equation}\label{eq:condition_on_coeff_MP}
\sum_{k_1,k_2\in\Z} \pr{\abs{k_1}+1}^2  \pr{\abs{k_2}+1}^2
a_{k_1,k_2}^2<+\infty,
\end{equation} 
the following estimate holds for $ \Delta'_n$:
 \begin{equation}
 \Delta'_n=O\pr{n^{-r} },\quad r:=\frac 12\min\ens{\delta,1-\frac 1{3+\delta}}.
 \end{equation}
In the context of Corollary~\ref{thm:vitesse_convergence_TLC}, the condition on the 
coefficients reads as follows:
\begin{equation}\label{eq:condition_on_coeff_G}
 \sum_{i=0}^{+\infty}\pr{\pr{i+1}^{1+\alpha}   +\pr{i+1}^{2-2/p+\beta}}
\pr{ \sum_{\pr{j_1,j_2}: \norm{\pr{j_1,j_2}}_\infty=i   } a_{j_1,j_2}^2 }^{1/2}<+\infty,
\end{equation}
where $p=2+\delta$.
 Let us compare \eqref{eq:condition_on_coeff_MP} with \eqref{eq:condition_on_coeff_G}. 
 Let $s:=\max\ens{1+\alpha,2-2/p+\beta}$. When $s\geq 2$,   
 \eqref{eq:condition_on_coeff_G} implies \eqref{eq:condition_on_coeff_MP}. 
 However, this implication does not hold if $s<3/2$. Indeed, let $r\in 
 \pr{s+1,5/2}$ and define $a_{k_1,k_2}:=k_1^{-r}$ if $k_1=k_2\geq 1$ and 
 $a_{k_1,k_2}:=0$ otherwise. Then  \eqref{eq:condition_on_coeff_G} holds 
 whereas \eqref{eq:condition_on_coeff_MP} does not.
 
 Let us discuss the convergence rates in the following example. Let $a_{k_1,k_2}:=
 2^{-\abs{k_1}-\abs{k_2}}$ and let $p=2+\delta$, where $\delta\in (0,1]$. 
 In our context, 
 \begin{equation}
 \abs{\frac{
   \abs{\Lambda_n\cap \pr{\Lambda_n-\gr{j}} }    }
 { \abs{ \Lambda_n }
 }-1 }\leq \frac{n^2-\pr{n-j_1}\pr{n-j_2}}{n^2}\leq \frac{j_1+j_2}n
 \end{equation}
 hence the convergence of $\sum_{j_1,j_2\in \Z}  
\abs{\operatorname{Cov}\pr{X_{0,0},X_{j_1,j_2}} } \pr{j_1+j_2} 
 $
 guarantees that $\varepsilon_n$ in 
 Corollary~\ref{thm:vitesse_convergence_TLC} is of order $1/n$. 
Moreover, since \eqref{eq:condition_on_coeff_G} holds for all $\alpha$ and 
$\beta$, the choice of $\gamma$ allows to reach rates of the form 
$n^{-\delta+r_0}$ for any fixed $r_0$. In particular, when $\delta=1$, one 
can reach for any fixed $r_0$ rates of the form $n^{-1+r_0}$. In comparison, 
with the same assumptions, the result of \cite{MR2805741} gives 
$n^{-3/8}$.

\section{Proofs}
 \label{sec:proofs}
 
 \subsection{Proof of Theorem~\ref{thm:moment_inequality_dependence_coefficients}}
 
 We define for $j\geq 1$ and $\gr{i}\in \Z^d$, 
 \begin{equation}
  X_{\gr{i},j}
  =\E{X_i \mid \sigma\pr{\eps_{\gr{u}}, \norm{\gr{u}-\gr{i}}_{\infty}
  \leq j}}-\E{X_i \mid \sigma\pr{\eps_{\gr{u}}, \norm{\gr{u}-\gr{i}}_{\infty}
  \leq j-1}}.
 \end{equation}
 In this way, by the martingale convergence theorem, 
 \begin{equation}\label{eq:decomposition_de_X_i}
  X_{\gr{i}}-\E{X_{\gr{i}} \mid \eps_{\gr{i}}}=
  \lim_{N\to +\infty}\sum_{j=1}^NX_{\gr{i},j}
 \end{equation}
 hence 
 \begin{equation}\label{eq:bound_proof_Wu0}
  \norm{\sum_{\gr{i}\in \Z^d}a_{\gr{i}}X_{\gr{i}}}_p
  \leq \sum_{j=1}^{+\infty}\norm{\sum_{\gr{i}\in \Z^d}a_{\gr{i}}X_{\gr{i},j}}_p
  +\norm{\sum_{\gr{i}\in \Z^d}a_{\gr{i}}\E{X_{\gr{i}} \mid \eps_{\gr{i}}}}_p.
 \end{equation}
 Let us fix $j\geq 1$.
We divide $\Z^d$ into blocks. For $\gr{v}\in \Z^d$, we define 
 \begin{equation}
  A_{\gr{v}}:=\prod_{q=1}^d 
  \pr{\left[\pr{2j+2}v_q,\pr{2j+2}\pr{v_q+1}-1 \right]\cap \Z},
 \end{equation}
and if $K$ is a subset of $[d]$, we define 
\begin{equation}
 E_K:=\ens{\gr{v}\in \Z^d, v_q\mbox{ is even if and only if }q\in K}.
\end{equation}
Therefore, the following inequality takes place 
\begin{equation}\label{eq:bound_proof_Wu1}
 \norm{\sum_{\gr{i}\in \Z^d}a_{\gr{i}}X_{\gr{i},j}}_p
 \leq \sum_{K\subset [d]}
\norm{\sum_{\gr{v}\in E_K}\sum_{\gr{i}\in A_{\gr{v}}}
a_{\gr{i}}X_{\gr{i},j}}_p.
\end{equation}

Observe that the random variable $\sum_{\gr{i}\in A_{\gr{v}}}
a_{\gr{i}}X_{\gr{i},j}$ is measurable for the $\sigma$-algebra 
generated by $\eps_{\gr{u}}$, where $\gr{u}$ satisfies 
$\pr{2j+2}v_q-\pr{j+1}\leq u_q\leq j+1+\pr{2j+2}\pr{v_q+1}-1$ 
for all $q\in [d]$. Since the family 
$\ens{\eps_{\gr{u}},\gr{u}\in \Z^d}$ is independent, the 
family $\ens{\sum_{\gr{i}\in A_{\gr{v}}}
a_{\gr{i}}X_{\gr{i},j}, \gr{v}\in E_K}$ is independent for each fixed
$K\subset [d]$. Using 
inequality \eqref{eq:Rosenthal}, it thus follows that 
\begin{multline} 
 \norm{\sum_{\gr{v}\in E_K}\sum_{\gr{i}\in A_{\gr{v}}}
a_{\gr{i}}X_{\gr{i},j}}_p 
\leq \frac{14.5p}{\log p}
\pr{\sum_{\gr{v}\in E_K} \norm{\sum_{\gr{i}\in A_{\gr{v}}}
a_{\gr{i}}X_{\gr{i},j}}_2^2}^{1/2}\\
+\frac{14.5p}{\log p}
\pr{\sum_{\gr{v}\in E_K} \norm{\sum_{\gr{i}\in A_{\gr{v}}}
a_{\gr{i}}X_{\gr{i},j}}_p^p}^{1/p}.
\end{multline}
By stationarity, one 
can see that $\norm{X_{\gr{i},j}}_q=\norm{X_{\gr{0},j}}_q$ for $q\in \ens{2,p}$, 
hence the triangle inequality yields 
\begin{multline}
 \norm{\sum_{\gr{v}\in E_K}\sum_{\gr{i}\in A_{\gr{v}}}
a_{\gr{i}}X_{\gr{i},j}}_p 
\leq \frac{14.5p}{\log p}\norm{X_{\gr{0},j}}_2
\pr{\sum_{\gr{v}\in E_K} \pr{\sum_{\gr{i}\in A_{\gr{v}}}
\abs{a_{\gr{i}} }}^2}^{1/2}\\
+\frac{14.5p}{\log p}\norm{X_{\gr{0},j}}_p
\pr{\sum_{\gr{v}\in E_K} \pr{\sum_{\gr{i}\in A_{\gr{v}}}
 \abs{a_{\gr{i}}} } ^p}^{1/p}.
\end{multline}
By Jensen's inequality, for $q\in \ens{2,p}$, 
\begin{equation}
  \pr{\sum_{\gr{i}\in A_{\gr{v}}}
\abs{a_{\gr{i}} }}^q \leq \abs{ A_{\gr{v}}  }^{q-1}
 \sum_{\gr{i}\in A_{\gr{v}}}
\abs{a_{\gr{i}}}^q\leq \pr{2j+2}^{d\pr{q-1}}\sum_{\gr{i}\in A_{\gr{v}}}
\abs{a_{\gr{i}}}^q
\end{equation}
and using $\sum_{i=1}^Nx_i^{1/q}\leq N^{\frac{q-1}q} \pr{\sum_{i=1}^Nx_i}^{1/q}$, 
it follows that 
\begin{multline}\label{eq:bound_proof_Wu2}
 \sum_{K\subset[d]}\norm{\sum_{\gr{v}\in E_K}\sum_{\gr{i}\in A_{\gr{v}}}
a_{\gr{i}}X_{\gr{i},j}}_p
\leq  \frac{14.5p}{\log p}\norm{X_{\gr{0},j}}_2
\pr{\sum_{\gr{i}\in \Z^d}a_i^2}^{1/2}\pr{4j+4}^{d/2}\\
+\frac{14.5p}{\log p}\norm{X_{\gr{0},j}}_p
\pr{\sum_{\gr{i}\in \Z^d}\abs{a_i}^p}^{1/p}\pr{4j+4}^{d\pr{1-1/p}}.
\end{multline}

Combining \eqref{eq:bound_proof_Wu0}, \eqref{eq:bound_proof_Wu1} and  \eqref{eq:bound_proof_Wu2}, we derive that 
\begin{multline}
\norm{\sum_{\gr{i}\in \Z^d}a_{\gr{i}}X_{\gr{i}}}_p
  \leq  \frac{14.5p}{\log p}\sum_{j=1}^{+\infty} \norm{X_{\gr{0},j}}_{2}
\pr{\sum_{\gr{i}\in \Z^d}a_i^2}^{1/2}\pr{4j+4}^{d/2}\\
+\frac{14.5p}{\log p}\sum_{j=1}^{+\infty} \norm{X_{\gr{0},j}}_{p}
\pr{\sum_{\gr{i}\in \Z^d}\abs{a_i}^p}^{1/p}\pr{4j+4}^{d\pr{1-1/p}}
  +\norm{\sum_{\gr{i}\in \Z^d}a_{\gr{i}}\E{X_{\gr{i}} \mid \eps_{\gr{i}}}}_p.
\end{multline}
In order to control the last term, we use inequality \eqref{eq:Rosenthal}
and bound $\norm{\E{X_{\gr{i}} \mid \eps_{\gr{i}}}}_q$ 
by $\norm{X_{\gr{0},0}}_{q}$ for $q\in\ens{1,2}$. This ends the proof of 
Theorem~\ref{thm:moment_inequality_dependence_coefficients}.
 
 \begin{proof}[Proof of Corollary~\ref{cor:moment_inequality_Wu}]
 The following lemma gives a control of the $\mathbb L^q$-norm of 
$X_{\gr{0},j}$ in terms of the physical measure dependence.
 
 \begin{Lemma}\label{lem:norm_Xoj_delta}
 For $q\in \ens{2,p}$ and $j\in \N$, the following inequality holds
 \begin{equation}\label{eq:bound_proof_Wu3}
  \norm{X_{\gr{0},j}}_q\leq 
  \pr{2\pr{q-1}\sum_{\gr{i}\in \Z^d,\norm{\gr{i}}_\infty=j} \delta_{\gr{i},q}^2}^{1/2}.
 \end{equation}

\end{Lemma}

\begin{proof}
Let $j$ be fixed. Let us write the set of elements of $\Z^d$ whose infinite norm 
is equal to $j$
as $\ens{\gr{v_s},1\leq s\leq N_j}$ where $N_j\in \N$. We also assume that 
$\gr{v_{s}}-\gr{v_{s-1}}\in\ens{\gr{e_k}, 1\leq k\leq d}$ for all $s\in\ens{2,\dots,N_j}$.

Denote 
\begin{equation}
 \f_s:=\sigma\pr{ \eps_{\gr{u}},\norm{\gr{u}}_\infty \leq j, 
 \eps_{\gr{v_t}},1\leq t\leq s},
\end{equation}
and $\f_0:=\sigma\pr{ \eps_{\gr{u}},\norm{\gr{u}}_\infty \leq j}$.
Then $X_{\gr{0},j}=\sum_{s=1}^{N_j}
\E{X_{\gr{0}} \mid \f_s}- 
\E{X_{\gr{0}} \mid \f_{s-1}}$, from which it follows, by Theorem~2.1
in \cite{MR2472010}, that 
\begin{equation}
 \norm{X_{\gr{0},j}}_q^2\leq \pr{q-1}\sum_{s=1}^{N_j}
\norm{\E{X_{\gr{0}} \mid \f_s}- 
\E{X_{\gr{0}} \mid \f_{s-1}}}_q^2.
\end{equation}
Then using similar arguments as in the proof of Theorem~1~(i) in \cite{MR2172215}
give the bound $\norm{\E{X_{\gr{0}} \mid \f_s}- 
\E{X_{\gr{0}} \mid \f_{s-1}}}_q\leq \delta_{\gr{v_s},q}+\delta_{\gr{v_{s-1}},q}$. 
This ends the proof of Lemma~\ref{lem:norm_Xoj_delta}.
\end{proof}
Now, Corollary~\ref{cor:moment_inequality_Wu} follows from an application 
of Lemma~\ref{lem:norm_Xoj_delta} with $q=2$ and $q=p$ respectively.
 \end{proof}
 
\subsection{Proof of Theorem~\ref{thm:vitesse_de_convergence_sommes_ponderees}}

Denote for a random variable $Z$ the quantity 
\begin{equation}
 \delta\pr{Z}:=\sup_{t\in \R}
 \abs{
 \mathbb P\ens{Z\leq t}-\Phi\pr{t}
 }.
\end{equation}

We say that a random field $\pr{Y_{\gr{i}}}_{\gr{i}\in \Z^d}$ 
is $m$-dependent if the collections of random variables 
$\pr{Y_{\gr{i}},\gr{i}\in A}$ and $\pr{Y_{\gr{i}},\gr{i}\in B}$ 
are independent whenever $\inf\ens{\norm{\gr{a}-\gr{b}}_{\infty},
\gr{a}\in A,\gr{b}\in B}>m$.
The proof of Theorem~\ref{thm:vitesse_de_convergence_sommes_ponderees} will use 
the following tools.
\begin{enumerate}[label=(T.\arabic*)]
 \item\label{itm:outil_1} By Theorem~2.6 in \cite{MR2073183}, if $I$ is a finite subset of $\Z^d$, 
 $\pr{Y_{\gr{i}}}_{\gr{i}\in I}$ an $m$-dependent centered random field such that 
 $\E{\abs{Y_{\gr{i}}}^p}<+\infty$ for each $\gr{i}\in I$ and some 
 $p\in (2,3]$ and $\operatorname{Var}\pr{\sum_{\gr{i}\in I}Y_{\gr{i}}
 }=1$, then 
 \begin{equation}
  \delta\pr{\sum_{\gr{i}\in I}Y_{\gr{i}    }}\leq 75\pr{10m+1}^{\pr{p-1}d}\sum_{\gr{i}\in I}
  \E{\abs{Y_{\gr{i}}}^p}.
 \end{equation}
\item\label{itm:outil_2} By Lemma~1 in \cite{MR2364223}, for any two random variables 
$Z$ and $Z'$ and $p\geq 1$, 
\begin{equation}
 \delta\pr{Z+Z'}\leq 2\delta\pr{Z}+\norm{Z'}_p^{\frac{p}{p+1}}.
\end{equation}

 \end{enumerate}

 Let $\pr{\eps_{\gr{u}}}_{\gr{u}\in \Z^d}$ be an i.i.d. random field 
 and let $f\colon \R^{\Z^d}\to \R$ be a measurable function such that 
 for each $\gr{i}\in \Z^d$, $X_{\gr{i}}=f\pr{\pr{\eps_{\gr{i}-\gr{u}}}_{\gr{u}\in\Z^d}}$. 
 Let $\gamma>0$ and $n_0$ defined by \eqref{eq:definition_de_n0}.
 
Let $m:= \pr{\left[\norm{b_n}_{\ell^2}\right]+1}^\gamma$ and let us define 
 \begin{equation}
  X_{\gr{i}}^{(m)}:=
  \E{X_{\gr{i}} \mid  
  \sigma\pr{\eps_{\gr{u}}, \gr{i}-m\gr{1}\imd \gr{u}\imd\gr{i}+ m\gr{1}}
  }.
 \end{equation}
 Since the random field $\pr{\eps_{\gr{u}}}_{\gr{u}\in \Z^d}$ is independent, the 
following properties hold.
\begin{enumerate}[label=(P.\arabic*)]
 \item\label{itm:propriete_1} The random field $\pr{X_{\gr{i}}^{(m)}}_{\gr{i}\in \Z^d}$ is 
 $\pr{2m+1}$-dependent.
 \item\label{itm:propriete_moments} The random field $\pr{X_{\gr{i}}^{(m)}}_{\gr{i}\in \Z^d}$ is identically 
 distributed and $\norm{X_{\gr{i}}^{(m)}}_{p'}
 \leq \norm{X_{\gr{0}}}_{p'}$.
 \item\label{itm:moments_inequality} For any $\pr{a_{\gr{i}}}_{\gr{i}\in \Z^d}\in \ell^2\pr{\Z^d}$ and 
 $q\geq 2$, the 
 following inequality holds:
\begin{multline}\label{eq:moment_inequality_approximation}
   \norm{\sum_{\gr{i}\in \Z^d} a_{\gr{i}}
   \pr{X_{\gr{i}}-X_{\gr{i}}^{(m)} }   }_q
   \leq 
  \frac{14.5q}{\log q}\pr{\sum_{\gr{i}\in \Z^d}a_{\gr{i}}^2}^{1/2}
   \sum_{j \geq m   } \pr{4j+4 }^{d/2}\norm{X_{\gr{0},j}}_2\\
   +  \frac{14.5q}{\log q} 
   \pr{\sum_{\gr{i}\in \Z^d}\abs{a_{\gr{i}}}^q}^{1/q}
   \sum_{j \geq m   }   
   \pr{4j+4 }^{d\pr{1-1/q}}\norm{X_{\gr{0},j}}_q.
  \end{multline}
  In order to prove \eqref{eq:moment_inequality_approximation}, 
  we follow the proof of Theorem~\ref{thm:moment_inequality_dependence_coefficients} 
  and start from the decomposition 
  $X_{\gr{i}}-X_{\gr{i}}^{(m)}=\lim_{N\to +\infty}\sum_{j=m}^NX_{\gr{i},j}$ 
  instead of \eqref{eq:decomposition_de_X_i}.
\end{enumerate}
Define $S_n^{(m)}:=\sum_{\gr{i}\in \Z^d}b_{n,\gr{i}}X_{\gr{i}}^{(m)}$.
An application of \ref{itm:outil_2} to $Z:= S_n^{(m)}\norm{b_n}_{\ell^2}^{-1}
\sigma^{-1}$ 
and $Z':=\pr{S_n-S_n^{(m)}}\norm{b_n}_{\ell^2}^{-1}\sigma^{-1}$ yields 
\begin{equation}
 \Delta_n\leq 2\delta\pr{\frac{S_n^{(m)}}{\sigma\norm{b_n}_{\ell^2}}}
 +\sigma^{-\frac{p}{p+1}}\frac 1{\norm{b_n}_{\ell^2}^{\frac{p}{p+1}}} 
 \norm{S_n-S_n^{(m)}}_p^{\frac{p}{p+1}}.
\end{equation}
Moreover, 
\begin{align}
 \delta\pr{\frac{S_n^{(m)}}{\sigma\norm{b_n}_{\ell^2}}}
 &=\sup_{t\in \R}
 \abs{\mathbb P\ens{\frac{S_n^{(m)}}{\sigma\norm{b_n}_{\ell^2}} \leq t }
 -\Phi\pr{t}
 }\\
 &=\sup_{u\in \R}
 \abs{\mathbb P\ens{\frac{S_n^{(m)}}{ \norm{S_n^{(m)} }_2} \leq u }
 -\Phi\pr{u\frac{\norm{S_n^{\pr{m}}}_2}{\sigma \norm{b_n}_{\ell^2}    }
 }}\\
 &\leq \delta\pr{\frac{S_n^{(m)}}{ \norm{S_n^{(m)} }_2}   }
 +\sup_{u\in \R}\abs{\Phi\pr{u\frac{\norm{S_n^{\pr{m}}}_2}{\sigma \norm{b_n}_{\ell^2} }}
 -\Phi\pr{u}},
\end{align}
hence, by \ref{itm:propriete_1} and \ref{itm:outil_1} applied with 
$Y_{\gr{i}}:=X_{\gr{i}}^{(m)}/\norm{S_n^{\pr{m}}}_2$, 
$p'$ instead of $p$ and $2m+1$ instead of $m$, we derive that 
\begin{equation}\label{eq:bound_Delta_n}
 \Delta_n\leq (I)+(II)+(III) 
\end{equation}
where 
\begin{equation}
 (I):=150\pr{20m+21}^{\pr{p'-1}d}\sum_{\gr{i}\in \Z^d}\abs{b_{n,\gr{i}}}^{p'}
 \norm{X_i^{\pr{m}}}_{p'}^{p'}\norm{S_n^{\pr{m}}}_{2}^{-p'},
\end{equation}
\begin{equation}
 (II):=2\sup_{u\in \R}
 \abs{\Phi\pr{u\frac{\norm{S_n^{\pr{m}}}_2}{\sigma \norm{b_n}_{\ell^2} }}-
 \Phi\pr{u}}
 \mbox{ and }
\end{equation}
\begin{equation}
 (III):=\sigma^{-\frac{p}{p+1}}\frac 1{\norm{b_n}_{\ell^2}^{\frac{p}{p+1}}} 
 \norm{S_n-S_n^{(m)}}_p^{\frac{p}{p+1}}.
\end{equation}
By \ref{itm:propriete_moments} and the reversed triangular inequality, 
the term $(I)$ can be bounded in the following way
\begin{equation}
 (I)\leq 150\pr{20m+21}^{\pr{p'-1}d}\norm{X_{\gr{0}}}_{p'}^{p'}
 \norm{b_n}_{\ell^{p'}}^{p'}
 \pr{\norm{S_n}_{2}-\norm{S_n-S_n^{\pr{m}}    }_2}^{-p'}
\end{equation}
and by \ref{itm:moments_inequality} with $q=2$, we obtain that 
\begin{equation}
 \pr{\norm{S_n}_{2}-\norm{S_n-S_n^{\pr{m}}    }_2}^{-p'} 
 \leq \pr{\norm{S_n}_{2}-29\pr{\log 2}^{-1} m^{-\alpha}\norm{b_n}_{\ell^2} C_2\pr{\alpha}       }^{-p'}.
\end{equation}
By \eqref{eq:definition_de_eps_n}, we have 
\begin{equation}
 \frac{\norm{S_n}_{2}^2}{\norm{b_n}_{\ell^2}^2}=\sigma^2+\eps_n,
\end{equation}
and we eventually get 
\begin{equation*}
 (I)\leq 150\pr{20m+21}^{\pr{p'-1}d}\norm{X_{\gr{0}}}_{p'}^{p'}
 \pr{\frac{\norm{b_n}_{\ell^{p'}} }{\norm{b_n}_{\ell^{2}} } }^{p'}
 \pr{\sqrt{\sigma^2+\eps_n}-29\pr{\log 2}^{-1} m^{-\alpha}C_2\pr{\alpha}   }^{-p'}.
\end{equation*}
Since $n\geq n_0$, we derive, in view of \eqref{eq:definition_de_n0}, 
\begin{equation}\label{eq:I}
 (I)\leq 150\pr{20m+21}^{\pr{p'-1}d}\norm{X_{\gr{0}}}_{p'}^{p'}
 \pr{\frac{\norm{b_n}_{\ell^{p'}} }{\norm{b_n}_{\ell^{2}} } }^{p'}
 \pr{\sigma/2  }^{-p'}
\end{equation}

In order to bound $(II)$, we argue as in \cite{MR2869760} (p. 456). Doing 
similar computations as in \cite{MR3225977} (p. 272), we obtain that 
\begin{equation}
 (II)\leq \pr{2\pi e}^{-1/2}\pr{\inf_{k\geq 1} a_k}^{-1}\abs{a_n^2-1},
\end{equation}
where $a_n:=\norm{S_n^{\pr{m}}}_2\sigma^{-1}\norm{b_n}_{\ell^2}^{-1}$. 
Observe that for any $n$, by \ref{itm:moments_inequality}, 
\begin{equation}
 a_n\geq \frac{\norm{S_n}_2-\norm{S_n-S_n^{\pr{m}}}_2}{\sigma\norm{b_n}_{\ell^2}}
 \geq \frac{\sqrt{\sigma^2+\eps_n}-29\pr{\log 2}^{-1}
  C_2\pr{\alpha} m^{-\alpha}   }{\sigma }
\end{equation}
and using again \ref{itm:moments_inequality} combined with 
Theorem~\ref{thm:moment_inequality_dependence_coefficients} for $p=q=2$,
\begin{align}
 \abs{a_n^2-1}&=\abs{\frac{\norm{S_n^{\pr{m}}}_2^2}{\sigma^2
 \norm{b_n}_{\ell^2}^2      }-1}\\
 &\leq \abs{\frac{\norm{S_n}_2^2}{\sigma^2
 \norm{b_n}_{\ell^2}^2      }-1}+
 \frac{\abs{\norm{S_n^{\pr{m}}}_2^2-\norm{S_n}_2^2}}{\sigma^2
 \norm{b_n}_{\ell^2}^2      }\\
 &\leq \frac{\abs{\eps_n}}{\sigma^2}+
  \frac{\abs{\norm{S_n^{\pr{m}}}_2- \norm{S_n}_2}\pr{
  \norm{S_n^{\pr{m}}}_2+ \norm{S_n}_2}    }{\sigma^2
 \norm{b_n}_{\ell^2}^2      }\\
 &\leq \frac{\abs{\eps_n}}{\sigma^2}+
  \frac{ \norm{S_n^{\pr{m}}- S_n}_2\pr{
  \norm{S_n^{\pr{m}}}_2+ \norm{S_n}_2}    }{\sigma^2
 \norm{b_n}_{\ell^2}^2      }\\
 &\leq \frac{\abs{\eps_n}}{\sigma^2}+40\pr{\log 2}^{-1}
 \frac{m^{-\alpha}}{\sigma^2}C_2\pr{\alpha}^2.
\end{align}
This leads to the estimate
\begin{equation}
 (II)\leq \frac{\pr{2\pi e}^{-1/2} }{\sqrt{\sigma^2+\eps_n}-29\pr{\log 2}^{-1}
  C_2\pr{\alpha} m^{-\alpha}   }\pr{\frac{\abs{\eps_n}}{\sigma}+40\pr{\log 2}^{-1}
 \frac{m^{-\alpha}}{\sigma}C_2\pr{\alpha}^2},
\end{equation}
and since $n\geq n_0$, we derive, in view of \eqref{eq:definition_de_n0}, 
\begin{equation}\label{eq:II}
 (II)\leq \pr{2\frac{\abs{\eps_n}}{\sigma^2}+80\pr{\log 2}^{-1}
 \frac{ \norm{b_n}_{\ell^2} ^{-\gamma\alpha}}{\sigma^2}C_2\pr{\alpha}^2}
 \pr{2\pi e}^{-1/2}.
\end{equation}

The estimate of $(III)$ relies on \ref{itm:moments_inequality}:
\begin{multline}
 (III)\leq \sigma^{-\frac{p}{p+1}}
  \pr{\frac{14.5p}{\log p} 
   \sum_{j \geq m   }\pr{4j+4 }^{d/2}
   \norm{X_{\gr{0},j}}_2}^{\frac{p}{p+1}}\\
   + \sigma^{-\frac{p}{p+1}}\norm{b_n}_{\ell^2}^{-\frac{p}{p+1}} 
   \norm{b_n}_{\ell^p}^{\frac{p}{p+1}}\pr{\frac{14.5p}{\log p} 
   \sum_{j \geq m   }   
   \pr{4j+4 }^{d\pr{1-1/p}}\norm{X_{\gr{0},j}}_p}^{\frac{p}{p+1}}
\end{multline}
 hence 
\begin{equation}\label{eq:III}
 (III)\leq\pr{ \frac{14.5p}{\sigma\log p} 4^{d/2}\norm{b_n}_{\ell^2}^{-\gamma\alpha}
 C_2\pr{\alpha}   }^{\frac{p}{p+1}}+
   \pr{\frac{\norm{b_n}_{\ell^p}}{\sigma \norm{b_n}_{\ell^2}   }\frac{14.5p}{\log p} 
   4^{d\pr{1-1/p}} \norm{b_n}_{\ell^2}^{-\gamma\beta}C_p\pr{\beta}
   }^{\frac{p}{p+1}} .
\end{equation}
The combination of \eqref{eq:bound_Delta_n}, \eqref{eq:I}, \eqref{eq:II} 
and \eqref{eq:III} gives \eqref{eq:Berry_Esseen_weighted_sums}.

 \subsection{Proof of Theorem~\ref{thm:regression}}

Since the random variables $X_{\gr{i}}$ are 
centered, we derive by definition of $g_n\pr{\gr{x}}$ that 
\begin{equation}\label{eq:reecriture_de_g_n}
 \pr{nh_n}^{d/2}\pr{g_n(\gr{x})-\E{g_n(\gr{x})}}=
 \pr{nh_n}^{d/2}\frac{\sum_{\gr{i}\in\Lambda_n}X_{\gr{i}} K\pr{\frac{\gr{x}-\gr{i}/n}{h_n}}}
 {\sum_{\gr{i}\in\Lambda_n} K\pr{\frac{\gr{x}-\gr{i}/n}{h_n}}}.
\end{equation}
We define 
\begin{equation}
 b_{n,\gr{i}}=K\pr{\frac 1{h_n}\pr{\gr{x}-\frac{\gr{i}}{n}}}, \quad\gr{i}
 \in \Lambda_n
\end{equation}
and $b_{n,\gr{i}}=0$ otherwise. Denote $b_n=\pr{b_{n,\gr{i}}}_{\gr{i}
\in\Z^d}$ and $\norm{b_n}_{\ell^2}
:=\pr{\sum_{\gr{i}\in \Z^d}b_{n,\gr{i}}^2  }^{1/2}$. In this way, 
by \eqref{eq:reecriture_de_g_n} and \eqref{eq:definition_de_An},  
\begin{equation}
 \frac 1{\norm{K}_{\mathbb L^2\pr{\R^d}} \sigma}
 \pr{nh_n}^{d/2}\pr{g_n(\gr{x})-\E{g_n(\gr{x})}}=\frac 1{\sigma}
 \sum_{\gr{i}\in \Z^d}b_{n,\gr{i}}X_{\gr{i}} 
 \norm{b_n}_{\ell^2}^{-1}A_n.
\end{equation}
Applying \ref{itm:outil_2} to $Z=\sum_{\gr{i}\in \Z^d}b_{n,\gr{i}}X_{\gr{i}} 
 \norm{b_n}_{\ell^2}^{-1}$ and 
 $Z'=\sum_{\gr{i}\in \Z^d}b_{n,\gr{i}}X_{\gr{i}} 
 \norm{b_n}_{\ell^2}^{-1}\sigma^{-1}\pr{A_n-1}$ and using 
 Theorem~\ref{thm:moment_inequality_dependence_coefficients}, we derive 
 that 
 \begin{equation}
  \widetilde{\Delta_n}\leq c_p\Delta'_n+c_p \pr{\sigma^{-1}
    C_2\pr{\alpha}+C_p\pr{\beta}    }^{\frac{p}{p+1}}
   \abs{A_n-1}^{\frac{p}{p+1}},
 \end{equation}
where 
\begin{equation}
 \Delta'_n=\sup_{t\in \R}\abs{\mathbb P\ens{
Z \leq t}
 -\Phi\pr{\frac{t}{\sigma }  }}.
\end{equation}
We then use Theorem~\ref{thm:vitesse_de_convergence_sommes_ponderees}
to handle $\Delta'_n$ (which is allowed, by \ref{itm:assumption1}).
Using boundedness of $K$, we control the $\ell^p$ and $\ell^{p'}$ norms
by a constant times the $\ell^2$-norm. This ends the proof of 
Theorem~\ref{thm:regression}.
 %%%%%%%%%%%%%%%%%%%%%%%%%%%%%%%%%%%%%%

\textbf{Acknowledgments} 
This research was supported by the grand SFB 823.

The author would like to thank the referees 
for many suggestions which improved the presentation of the paper.
 
\def\polhk\#1{\setbox0=\hbox{\#1}{{\o}oalign{\hidewidth
  \lower1.5ex\hbox{`}\hidewidth\crcr\unhbox0}}}\def\cprime{$'$}
  \def\polhk#1{\setbox0=\hbox{#1}{\ooalign{\hidewidth
  \lower1.5ex\hbox{`}\hidewidth\crcr\unhbox0}}} \def\cprime{$'$}
\providecommand{\bysame}{\leavevmode\hbox to3em{\hrulefill}\thinspace}
\providecommand{\MR}{\relax\ifhmode\unskip\space\fi MR }
\providecommand{\MRhref}[2]{%
  \href{http://www.ams.org/mathscinet-getitem?mr=#1}{#2}
}
\providecommand{\href}[2]{#2}

\end{document}